\documentclass{amsart}

\usepackage{mathrsfs}
\usepackage{amsfonts}
\usepackage{amsthm}
\usepackage{amsmath}
\usepackage{url}

\newtheorem{thm}{Theorem}[section]
\newtheorem{lem}[thm]{Lemma}
\newtheorem{cor}[thm]{Corollary}

\newtheorem{fact}{Theorem}

\theoremstyle{remark}
\newtheorem{rmk}[thm]{Remark}

\theoremstyle{definition}
\newtheorem{definition}[thm]{Definition}
\newtheorem{ex}[thm]{Example}

\numberwithin{equation}{section}

\DeclareMathOperator{\ess}{ess}
\DeclareMathOperator{\dist}{dist}
\DeclareMathOperator{\ind}{ind}

\newcommand{\cR}{\mathcal{R}}
\newcommand{\cM}{\mathcal{M}}
\newcommand{\cV}{\mathcal{V}}
\newcommand{\cW}{\mathcal{W}}
\newcommand{\cS}{\mathcal{S}}

\newcommand{\C}{\mathbb{C}}
\newcommand{\D}{\mathbb{D}}
\newcommand{\M}{\mathbb{M}}

\newcommand{\R}{\mathbb{R}}
\newcommand{\T}{\mathbb{T}}
\newcommand{\pP}{\mathbb{P}}

\newcommand{\Zero}{\mathbb{O}}
\newcommand{\Mydef}{\stackrel{\mbox{\footnotesize{\rm def}}}{=}}

\title{Monotone thematic factorizations of matrix functions}

\author{Alberto A. Condori}

\begin{document}
\address{Department of Mathematics, Michigan State University, East Lansing, MI 48824, USA}
\email{condoria@msu.edu}


\begin{abstract}
			We continue the study of the so-called thematic factorizations of admissible very badly approximable matrix functions.  
			These factorizations were introduced by V.V. Peller and N.J. Young for studying superoptimal approximation by bounded 
			analytic matrix functions.  Even though thematic indices associated with a thematic factorization of an admissible very 
			badly approximable matrix function are not uniquely determined by the function itself, R.B. Alexeev and V.V. Peller 
			showed that the thematic indices of any monotone \emph{non-increasing} thematic factorization of an admissible very badly 
			approximable matrix function are uniquely determined.  In this paper, we prove the existence of monotone \emph{non-decreasing} 
			thematic factorizations for admissible very badly approximable matrix functions. It is also shown that the thematic indices 
			appearing in a monotone non-decreasing thematic factorization are \emph{not} uniquely determined by the matrix function itself.  
			Furthermore, we show that the monotone non-increasing thematic factorization gives rise to a great number of other thematic 
			factorizations.  
\end{abstract}
\maketitle

\section{Introduction}\label{intro}

The problem of approximating a function by bounded analytic functions on the unit circle $\T$ under the uniform norm has been 
studied by many mathematicians.  In \cite{Kh}, it was shown that any continuous function $\varphi$ on $\T$ has a unique best 
approximation $f$ by bounded analytic functions, i.e. $\|\varphi-f\|_{\infty}=\dist(\varphi, H^{\infty})$, and the error 
function $\varphi-f$ has constant modulus on $\T$.  Different authors have studied the error function, or equivalently, 
functions $\psi$ for which the zero function is a best approximation.  These functions $\psi$ are called \emph{badly 
approximable}. For example, it was proved in \cite{Po} that a continuous function $\psi$ is badly approximable if and only 
if it has constant modulus and negative winding number.  This classification was extended to a larger collection of functions 
for which the notion of winding number is not available and can instead be stated in terms of Hankel and Toeplitz operators.  
This collection is referred to as the class of admissible functions.  A function $\psi\in L^{\infty}$ is said to be 
\emph{admissible} if the essential norm $\|H_{\psi}\|_{\rm e}$ of the Hankel operator $H_{\psi}$ is strictly less than its 
operator norm $\|H_{\psi}\|$.  It is well-known now that an admissible function $\psi$ is badly approximable if and only if 
$\psi$ has constant modulus, the Toeplitz operator $T_{\psi}$ is Fredholm, and \mbox{$\ind T_{\psi}>0$} (recall that for a 
Fredholm operator $T$, its index, $\ind T$, is defined to be $\dim\ker T-\dim\ker T^{*}$).  Moreover, it was proved by Nehari
(e.g. see Chapter 1 in \cite{Pe}) that
\[	\dist_{L^{\infty}}(\psi,H^{\infty})=\|H_{\psi}\|.	\]
As usual, given a bounded function $\psi$, the Hankel operator $H_{\psi}: H^{2}\rightarrow H^{2}_{-}\Mydef L^{2}\ominus H^{2}$ 
and Toeplitz operator $T_{\psi}: H^{2}\rightarrow H^{2}$ are defined by the formulas
\[	H_{\psi}f\Mydef\pP_{-}\psi f\,\mbox{ and }\; T_{\psi}f\Mydef\pP_{+}\psi f,	\]
respectively, where $\pP_{+}$ denotes the orthogonal projection from $L^{2}$ onto $H^{2}$ and $\pP_{-}=I_{L^{2}}-\pP_{+}$ 
is the orthogonal projection from $L^{2}$ onto $H^{2}_{-}$.

In this note, we deal mainly with matrix-valued functions on $\T$. We begin by introducing notation, definitions and some
known facts in order to adequately explain our results.\smallskip

\subsection{Badly and very badly approximable matrix functions}

Let $\M_{m,n}$ denote the space of $m\times n$ matrices equipped with the operator norm $\|\cdot\|_{\M_{m,n}}$ (of the 
space of linear operators from $\C^{n}$ to $\C^{m}$).  In the case of $n\times n$ matrices, we use the notation $\M_{n}$ 
to denote $\M_{n,n}$.  For $A\in\M_{m,n}$, we denote by $s_{j}(A)$ the $j$th-\emph{singular value} of the matrix 
$A$, i.e. the distance from $A$ (under the operator norm) to the set of matrices of rank at most $j$, for $j\geq 0$.

For an $m\times n$ matrix-valued function $G$ on $\T$, we define	
\[	\|G\|_{L^{\infty}(\M_{m,n})}=\ess\sup_{\zeta\in\T}\|G(\zeta)\|_{\M_{m,n}}.	\]
Also, for a space $X$ of scalar functions defined on $\T$, we use the notation $X(\M_{m,n})$ to denote $m\times n$ 
matrix functions whose entries belong to $X$.  In particular, if $n=1$, we use the notation $X(\C^{m})$ to denote 
$X(\M_{m,1})$.

\begin{definition}
			Given an $m\times n$ matrix-valued function $\Phi\in L^{\infty}(\M_{m,n})$, we say that $F\in H^{\infty}(\M_{m,n})$ 
			is a \emph{best approximation of $\Phi$} if
			\[	\|\Phi-F\|_{L^{\infty}(\M_{m,n})}=\dist_{L^{\infty}(\M_{m,n})}(\Phi, H^{\infty}(\M_{m,n})).	\]
			If the zero matrix function $\Zero$ is a best approximation of $\Phi$, then $\Phi$ is said to be \emph{badly approximable}.
\end{definition}

It is well-known that, unlike the scalar case, the continuity of a matrix function $\Phi$ does not guarantee the uniqueness 
of a best approximation.  However, the continuity of a matrix function $\Phi$ does guarantee uniqueness if the criterion of 
optimality of an approximation $F$ to $\Phi$ is refined as follows.

\begin{definition}\label{superoptDef}
			Let $\Phi\in L^{\infty}(\M_{m,n})$.  For $k\geq 0$, define the sets $\Omega_{k}=\Omega_{k}(\Phi)$ by
			\begin{align*}
						\Omega_{0}(\Phi)&=\{	F\in H^{\infty}(\M_{m,n}): F \mbox{ minimizes }\ess\sup_{\zeta\in\T}\|\Phi(\zeta)-F(\zeta)\|_{\M_{m,n}}	\},
														\mbox{ and}\\
						\Omega_{j}(\Phi)&=\{	F\in\Omega_{j-1}: F \mbox{ minimizes }\ess\sup_{\zeta\in\T}s_{j}(\Phi(\zeta)-F(\zeta))	\}\mbox{ for }j>0.
			\end{align*}
			We say that $F$ is a \emph{superoptimal approximation of $\Phi$} by bounded analytic matrix functions if $F$ belongs to
			$\displaystyle{\bigcap_{k\geq0}\Omega_{k}=\Omega_{\min\{m,n\}-1}}$, and in this case we define the \emph{superoptimal
			singular values of $\Phi$} by 
			\[	t_{j}(\Phi)=\ess\sup_{\zeta\in\T}s_{j}((\Phi-F)(\zeta))\mbox{ for }j\geq 0.	\]
			If the zero matrix function $\Zero$ belongs to $\Omega_{\min\{m,n\}-1}$, we say that $\Phi$ is \emph{very 
			badly approximable}.
\end{definition}

It follows from Definition \ref{superoptDef} that a matrix function $F$ belongs to $\Omega_{0}(\Phi)$ if and only if $F$ is a best 
approximation of $\Phi$.  In addition, $F$ is a superoptimal approximation of $\Phi$ if and only if $\Phi-F$ is very badly approximable.

As in the case of best approximation of bounded scalar-valued functions, Hankel operators on Hardy spaces play a 
fundamental role in the study of superoptimal approximation.  For a matrix function $\Phi\in L^{\infty}(\M_{m,n})$, 
we define the \emph{Hankel operator} $H_{\Phi}$ by
\[	H_{\Phi}f=\pP_{-}\Phi f\mbox{ for }f\in H^{2}(\C^{n}),	\]
where $\pP_{-}$ denotes the orthogonal projection of $L^{2}(\C^{m})$ onto $H^{2}_{-}(\C^{m})\Mydef L^{2}(\C^{m})\ominus 
H^{2}(\C^{m})$.  By a matrix analog of Nehari's Theorem (e.g. Theorem 2.2.2 in \cite{Pe}), it is known that
\[	\dist_{L^{\infty}(\M_{m,n})}(\Phi, H^{\infty}(\M_{m,n}))=\|H_{\Phi}\|.	\]
In particular, a matrix function $\Phi$ is badly approximable if and only if $\|H_{\Phi}\|=\|\Phi\|_{L^{\infty}(M_{m,n})}$.
Furthermore, it is known that
\[	\dist_{L^{\infty}(\M_{m,n})}(\Phi, (H^{\infty}+C)(\M_{m,n}))=\|H_{\Phi}\|_{\rm e}	\]
holds (e.g. see Theorem 4.3.8 in \cite{Pe}),  recalling that $H^{\infty}+C$ denotes the set of scalar functions in $L^{\infty}$ 
which are a sum of a bounded analytic function on the unit disk $\D$ and a continuous function on $\T$, and the essential norm 
of an operator $T$ between two Hilbert spaces is defined to be 
\[	\|T\|_{\rm e}\Mydef\inf\{\|T-K\|: K\mbox{ is compact }\}.\]

In this note, we consider matrix functions $\Phi\in L^{\infty}(\M_{m,n})$ such that $\|H_{\Phi}\|_{\rm e}$ is strictly less 
than the smallest non-zero superoptimal singular value of $\Phi$.  We call such matrix functions $\Phi$ \emph{admissible}. 
Observe that this definition is a generalization of the previously mentioned notion of admissibility for scalar-valued functions.
It is easy to see that any matrix function $\Phi\in (H^{\infty}+C)(\M_{m,n})\setminus H^{\infty}(\M_{m,n})$ is admissible
because the essential norm of $H_{\Phi}$ is zero in this case.

To date, the class of admissible bounded matrix functions is the largest for which uniqueness of a superoptimal approximation 
by bounded analytic matrix functions is guaranteed.  This was proved first in \cite{PY1} for the special case of matrix functions 
in $(H^{\infty}+C)(\M_{m,n})\setminus H^{\infty}(\M_{m,n})$, and shortly after in \cite{PT1} for admissible matrix functions.

\subsection{Balanced matrix functions and thematic factorizations}\label{thematicSection}

In \cite{PY1}, very badly approximable matrix functions in $(H^{\infty}+C)(\M_{m,n})$ were characterized algebraically in 
terms of thematic factorizations.  It turns out that this same algebraic characterization remains valid for very badly approximable 
matrix functions which are only admissible.  We first recall several definitions to discuss these factorizations.

Let $I_{n}$ denote the matrix function that equals the $n\times n$ identity matrix on $\T$.  Recall that a matrix function 
$\Theta\in H^{\infty}(\M_{m,n})$ is called \emph{inner} if $\Theta^{*}\Theta=I_{n}$ a.e. on $\T$.  A matrix function $\Upsilon
\in H^{\infty}(\M_{m,n})$ is called \emph{outer} if $\Upsilon H^{2}(\C^{n})$ is dense in $H^{2}(\C^{m})$.  Lastly, a matrix 
function $F\in H^{\infty}(\M_{m,n})$ is called  \emph{co-outer} whenever the transposed function $F^{t}$ is outer.

Let $n\geq2$ and $0<r<n$.  For an $n\times r$ inner and co-outer matrix function  $\Upsilon$, it is known that there is 
an $n\times(n-r)$ inner and co-outer matrix function $\Theta$ such that
\begin{equation}\label{themFunct}
			V=(\,\Upsilon\;\overline{\Theta}\,)
\end{equation} 
is a unitary-valued matrix function on $\T$.  Functions of the form $(\ref{themFunct})$ are called \emph{r-balanced}.  We 
refer the reader to Chapter 14 in \cite{Pe} for a detailed presentation of many interesting properties of $r$-balanced matrix 
functions.  

Our main interest lies with 1-balanced matrix functions, which are also referred to as \emph{thematic}.  

\begin{definition}
			A \emph{partial thematic factorization} of an $m\times n$ matrix function is a factorization of the form
			\begin{equation}\label{partialFact}
						W^{*}_{0}\cdot\ldots\cdot W^{*}_{r-1}\left(
						\begin{array}{ccccc}
									t_{0}u_{0}	&	\Zero				&	\ldots	&	\Zero						&	\Zero\\
									\Zero				&	t_{1}u_{1}	&	\ldots	&	\Zero						&	\Zero\\
									\vdots			&	\vdots			&	\ddots	&	\vdots					&	\vdots\\
									\Zero				&	\Zero				&	\ldots	&	t_{r-1}u_{r-1}	&	\Zero\\
									\Zero				&	\Zero				&	\ldots	&	\Zero						&	\Psi
						\end{array}\right)V^{*}_{r-1}\cdot\ldots\cdot V^{*}_{0}
			\end{equation}
			where the numbers $t_{0},t_{1},\ldots,t_{r-1}$ satisfy 
			\[	t_{0}\geq t_{1}\geq\ldots\geq t_{r-1}> 0;	\]
			the function $u_{j}$ is unimodular and such that the Toeplitz operator $T_{u_{j}}$ is Fredholm with positive index,
			for $0\leq j\leq r-1$; the $n\times n$ matrix function $V_{j}$ and $m\times m$ matrix function $W_{j}$ have the form
			\begin{equation}\label{almostThematic}
						V_{j}=\left(\begin{array}{cc}
															I_{j}	&	\Zero\\
															\Zero	&	\breve{V}_{j}
												\end{array}\right)\,\mbox{ and }\;
						W_{j}=\left(\begin{array}{cc}
															I_{j}	&	\Zero\\
															\Zero	&	\breve{W}_{j}
												\end{array}\right),
			\end{equation}
			for some thematic matrix functions $\breve{V}_{j}$ and $\breve{W}_{j}^{t}$, respectively, for $1\leq j<r-1$;  
			$V_{0}$ and $W_{0}^{t}$ are thematic matrix functions; and the matrix function $\Psi$ satisfies 
			\begin{equation}\label{PsiCondition}
						\|\Psi\|_{L^{\infty}(\M_{m-r,n-r})}\leq t_{r-1}\,	\mbox{ and }\;\|H_{\Psi}\|<t_{r-1}.
			\end{equation}
			The positive integers $k_{0},\ldots, k_{r-1}$ defined by
			\[	k_{j}=\ind T_{u_{j}},\;\mbox{ for }0\leq j\leq r-1,	\]
			are called the \emph{thematic indices} associated with the factorization in $(\ref{partialFact})$.		
\end{definition}

As usual, if $r=m$ or $r=n$, we use the convention that the corresponding row or column does not exist.

\begin{definition}
			A \emph{thematic factorization} of an $m\times n$ matrix function is a partial thematic factorization of the form 
			$(\ref{partialFact})$ in which $\Psi$ is identically zero.
\end{definition}

It can be shown that any admissible very badly approximable matrix function admits a thematic factorization. Conversely, any 
matrix function of the form $(\ref{partialFact})$ with $\Psi=\Zero$ is a very badly approximable matrix function whose 
$j$th-superoptimal singular value equals $t_{j}$ for $0\leq j\leq r-1$.  See Chapter 14 of \cite{Pe} for proofs of these results.

In \cite{PY1}, it was observed that thematic indices depend on the choice of the thematic factorization.  However, it 
was conjectured there that the sum of the thematic indices associated with any thematic factorization of a given very badly 
approximable matrix function $\Phi$ depends only on $\Phi$ (and is therefore independent of the choice of a thematic 
factorization) whenever $\Phi$ belongs to $(H^{\infty}+C)(\M_{m,n})$.  This conjecture was settled in the affirmative 
shortly after in \cite{PY2}.  Moreover, it was shown in \cite{PT1} that this conjecture remains valid for matrix functions 
$\Phi$ which are merely admissible.

The result concerning the sum of thematic indices of $\Phi$ leads to the question:  \emph{Can one arbitrarily distribute this 
sum among thematic indices of $\Phi$ by choosing an appropriate thematic factorization?} A partial answer was given in \cite{AP2} 
in terms of monotone partial thematic factorizations.

\begin{definition}
			A partial thematic factorization of the form $(\ref{partialFact})$ is called \emph{monotone non-increasing} (or non-decreasing) 
			if for any superoptimal singular value $t$, such that $t\geq t_{r-1}$, the thematic indices $k_{j},k_{j+1},\ldots,k_{s}$ that 
			correspond to all of the superoptimal singular values that are equal to $t$ form a monotone non-increasing sequence 
			(or non-decreasing sequence).
\end{definition}

\begin{rmk}
			Note that only monotone \emph{non-increasing} partial thematic factorizations were considered in \cite{AP2}.
\end{rmk}

The following result was established in \cite{AP2}.

\begin{fact}
			If $\Phi\in L^{\infty}(\M_{m,n})$ is an admissible very badly approximable matrix function,
			then $\Phi$ possesses a \emph{monotone non-increasing} thematic factorization.  Moreover, the indices 
			of any monotone non-increasing thematic factorization are uniquely determined by $\Phi$.
\end{fact}

Hence, one cannot arbitrarily distribute the sum of thematic indices of an admissible very badly approximable matrix function 
among thematic indices in \emph{non-increasing order}.  Indeed, thematic indices are uniquely determined when arranged in this way.

We refer the reader to \cite{Pe} for more information and proofs of all previously mentioned facts concerning Hankel 
operators, superoptimal approximation, and thematic factorizations of admissible very badly approximable matrix functions.  
See also \cite{AP2} and \cite{PT2} for other characterizations of badly and very badly approximable matrix functions.

\subsection{What is done in this note?}

Consider now the following example.  Let $G$ be the $2\times 2$-matrix function defined by
\begin{equation*}
			G=\frac{1}{\sqrt{2}}\left(	\begin{array}{cc}
																				\bar{z}	&	-1\\
																				1				&	z
																	\end{array}\right)\left(
																	\begin{array}{cc}
																				\bar{z}^{2}	&	\Zero\\
																				\Zero				&	\bar{z}
																	\end{array}\right)=\frac{1}{\sqrt{2}}\left(
																	\begin{array}{cc}
																				\bar{z}^{3}	&	-\bar{z}\\
																				\bar{z}^{2}	&	1
																	\end{array}\right).
\end{equation*}
Clearly, $G$ is a very badly approximable continuous (and so admissible) function in its non-increasing monotone thematic 
factorization with thematic indices $2$ and $1$.  We now ask the question:  Does $G$ admit a monotone \emph{non-decreasing} 
thematic factorization?  

It is easy to verify that $G$ can also be factored as
\begin{equation*}
			G=\left(\begin{array}{cc}
										-1		&	\Zero\\
										\Zero	&	1
							\end{array}\right)
			\left(\begin{array}{cc}
										\bar{z}	&	\Zero\\
										\Zero		&	\bar{z}^{2}
							\end{array}\right)\frac{1}{\sqrt{2}}\left(
							\begin{array}{cc}
										-\bar{z}^{2}	&	1\\
										1						&	z^{2}
							\end{array}\right)
\end{equation*}
demonstrating that $G$ does admit a monotone non-decreasing thematic factorization with thematic indices $1$ and $2$.
Thus,  the natural question arises: \emph{Does every admissible very badly approximable matrix function admit a monotone 
\emph{non-decreasing} thematic factorization?  If so, are the thematic indices in any such factorization uniquely 
determined by the matrix function itself?}  We succeed in providing answers to these questions.

We begin Section $\ref{SectionScalar}$ introducing sufficient conditions under which the Toeplitz operator induced by a 
unimodular function is invertible.  For the reader's convenience, we also state some well-known theorems on the factorization 
of certain unimodular functions. 

In Section $\ref{SectionBadlyApp}$, we establish new results on badly approximable matrix functions.  We prove that given a 
(partial) thematic factorization of a badly approximable matrix function $G$ whose ``second'' thematic index equals $k$ and 
an integer $j$ satisfying $1\leq j\leq k$, it is possible to find a new (partial) thematic factorization of $G$ in which the 
``first'' new thematic index equals $j$.  We then give further analysis of the ``lower block'' obtained in this new 
factorization of $G$.  It is shown that, under rather natural assumptions, the first thematic index of the new lower block 
is indeed the first thematic index of $G$ in the originally given thematic factorization. 

Once these results are available, we argue in Section $\ref{SectionMonoFactorization}$ that there is an abundant number of 
thematic factorizations of an arbitrary (admissible) very badly approximable matrix function.  We begin by proving the existence 
of a monotone non-decreasing thematic factorization for such matrix functions.  In contrast to monotone non-increasing thematic 
factorizations, it is shown that the thematic indices appearing in a monotone non-decreasing thematic factorization are \emph{not} 
uniquely determined by the matrix function itself.  Moreover, we obtain every possible sequence of thematic indices in the case 
of $2\times 2$ unitary-valued matrix functions.  We further prove that one can obtain various thematic factorizations from a 
monotone non-increasing thematic factorization while preserving ``some structure'' of the thematic indices in the case of 
$m\times n$ matrix functions with $\min\{m,n\}\geq2$.  We close the section by illustrating this with a simple example.

In Section \ref{unitarySection}, we provide an algorithm and demonstrate with an example that the algorithm yields a thematic 
factorization for any specified sequence of thematic indices of an arbitrary admissible very badly approximable unitary-valued 
$2\times 2$ matrix function.

\section{Invertibility of Toeplitz operators and factorization of certain unimodular functions}\label{SectionScalar}

In this section, we include some useful and perhaps well-known (to those who work with Toeplitz and Hankel operators on the Hardy 
space $H^{2}$) results regarding scalar functions that are needed throughout the paper.  We begin by introducing sufficient conditions
for which a Toeplitz operator $T_{w}$, where $w$ is a unimodular function on $\T$ (i.e. $w$ has modulus equal to 1 a.e. on $\T$),
is invertible on $H^2$.  Although a complete description of unimodular functions $w$ for which $T_{w}$ is invertible is given by the 
well-known theorem of Devinatz and Widom, the sufficient condition given in Theorem \ref{invThm} below is easier to verify.

\begin{lem}\label{scalarLemma}
			Let $0<p\leq\infty$.  If $h\in H^{p}$ and $1/h\in H^{2}$, then the Toeplitz operator $T_{\bar{h}/h}$ has trivial kernel.
\end{lem}
\begin{proof}
			Suppose that $p\geq 2$.  Let $f\in \ker T_{\bar{h}/h}$.  Since $H^{2}_{-}= L^{2}\ominus H^{2}=\bar{z}\overline{H^{2}}$, 
			then $f/h\in (1/\bar{h})\bar{z}\overline{H^{2}}$.  It follows that $f/h\in H^{1}\cap\bar{z}\overline{H^{1}}$ and therefore
			$\ker T_{\bar{h}/h}$ must be trivial, because $H^{1}\cap\bar{z}\overline{H^{1}}$ is trivial.
			
			Suppose now that $h\in H^{p}\setminus H^{2}$ with $0<p<2$.  Assume, for the sake of contradiction, that $\ker T_{\bar{h}/h}$ 
			is non-trivial.  In this case, a simple argument of Hayashi (see the proof of Lemma 5 in \cite{Ha}) shows that there is 
			an outer function $k\in H^{2}$ such that $\bar{h}/h=\bar{k}/k$, and so there is a $c\in\R$ such that $h=ck$, a contradiction
			to the assumption that $h\notin H^{2}$.  Thus $T_{\bar{h}/h}$ must have trivial kernel.
\end{proof}

\begin{thm}\label{invThm}
			Suppose that $h\in H^{2}$ and $1/h\in H^{2}$.  Then the Toeplitz operator $T_{\bar{h}/h}$ has trivial kernel and dense range.
			In particular, if $T_{\bar{h}/h}$ is Fredholm, then $T_{\bar{h}/h}$ is invertible.
\end{thm}
\begin{proof}
			By Lemma \ref{scalarLemma}, we know that $T_{\bar{h}/h}$ has trivial kernel.   Now, $h\in H^{2}$ and $1/h\in H^{2}$ imply 
			that $h$ is an outer function, and so the fact that $T_{\bar{h}/h}$ has dense range follows from Theorem 4.4.10 in \cite{Pe}. 
			The rest is obvious.
\end{proof}

We now state a useful converse to Theorem \ref{invThm}.

\begin{fact}\label{DWfact}
			If $w$ is a unimodular function on $\T$ such that $T_{w}$ is invertible on $H^{2}$, then $w$ admits a factorization of the 
			form $w=\bar{h}/h$ for some outer function $h$ such that both $h$ and $1/h$ belong to $H^{p}$ for some $2<p\leq\infty$.
\end{fact}

This result can be deduced from the theorem of Devinatz and Widom mentioned earlier.  A proof can be found in Chapter 3 of \cite{Pe}.

We now state two useful, albeit immediate, implications of Fact \ref{DWfact}.

\begin{cor}
			Suppose that $h$ and $1/h$ belong to $H^{2}$.  If the Toeplitz operator $T_{\bar{h}/h}$ is Fredholm, then $h$ and $1/h$ 
			belong to $H^{p}$ for some $2<p\leq\infty$.
\end{cor}

\begin{cor}\label{badFact}
			Let $u$ be a unimodular function on $\T$.  If the Toeplitz operator $T_{u}$ is Fredholm with index $k$, then there is an
			outer function $h$ such that 
			\begin{equation}\label{scalarRep}
						u=\bar{z}^{k}\frac{\bar{h}}{h}
			\end{equation}
			and both $h$ and $1/h$ belong to $H^{p}$ for some $2<p\leq\infty$.
\end{cor}

\begin{rmk}\label{classR}
			Even though representation $(\ref{scalarRep})$ is very useful (e.g. in the proof of Theorem \ref{MainThm}), it may be 
			difficult to find the function $h$ explicitly, if needed.  This is however a very easy task for unimodular functions 
			in the space $\cR$ of rational functions with poles outside of $\T$.  After all, if $u\in\cR$, then there are finite 
			Blaschke products $B_{1}$ and $B_{2}$ such that $u=\bar{B}_{1}B_{2}$, by the Maximum Modulus Principle.  Thus, $u$ 
			admits a representation of the form $(\ref{scalarRep})$ with $k=\deg B_{1}-\deg B_{2}$ for some function $h$ invertible 
			in $H^{\infty}$ (which is, up to a multiplicative constant, a product of quotients of reproducing kernels of $H^{2}$).
\end{rmk}

We also find the classification of admissible scalar badly approximable functions mentioned in Section \ref{intro} and 
Remark \ref{classR} useful in proving the next theorem which is part of the lore of our subject.

\begin{thm}\label{scalarFact}
			Suppose that $u\in\cR$ is a unimodular function on $\T$. Then $u$ is badly approximable if and only if there are 
			finite Blaschke products $B_{1}$ and $B_{2}$ such that $\deg B_{1}>\deg B_{2}$ and $u=\bar{B}_{1}B_{2}$ on $\T$.  
			In particular, $u$ admits the representation
			\[	u=\bar{z}^{k}\frac{\bar{h}}{h}	\]
			with $k=\ind T_{u}=\deg B_{1}-\deg B_{2}$ for some function $h$ invertible in $H^{\infty}$.
\end{thm}

\section{Badly approximable matrix functions}\label{SectionBadlyApp}

Recall that for $T:X\rightarrow Y$, a bounded linear operator between normed spaces $X$ and $Y$, a vector $x\in X$ is 
called a \emph{maximizing vector of $T$} if $x$ is non-zero and $\|Tx\|_{Y}=\|T\|\cdot\|x\|_{X}$.

\begin{definition}\label{mDef}
			For a matrix function $\Phi\in L^{\infty}(\M_{m,n})$ such that $\|H_{\Phi}\|_{\rm e}<\|H_{\Phi}\|$, we define \emph{the 
			space $\cM_{\Phi}$ of maximizing vectors of $H_{\Phi}$} by
			\[	\cM_{\Phi}\Mydef\{	f\in H^{2}(\C^{n}): \|H_{\Phi}f\|_{2}=\|H_{\Phi}\|\cdot\|f\|_{2}	\}.	\]
\end{definition}
It is easy to show that $M_{\Phi}$ is a closed subspace which consists of the zero vector and all maximizing vectors 
of the Hankel operator $H_{\Phi}$.  Moreover, $\cM_{\Phi}$ always contains a maximizing vector of $H_{\Phi}$ because 
$\|H_{\Phi}\|_{\rm e}<\|H_{\Phi}\|$;  a consequence of the spectral theorem for bounded self-adjoint operators.

We now review results concerning badly approximable matrix functions that are used in this section.  Let 
$G\in L^{\infty}(\M_{m,n})$ be a badly approximable function such that $\|H_{G}\|_{\rm e}<1$ and $\|H_{G}\|=1$. 
In this case, it is not difficult to show that if $f$ is a non-zero function in $\cM_{G}$, then $G f\in H^{2}_{-}(\C^{m})$, 
$\|G(\zeta)\|_{\M_{m,n}}=1$ for a.e. $\zeta\in\T$, and $f(\zeta)$ is a maximizing vector of $G(\zeta)$ for a.e. $\zeta\in\T$
(see Theorem 3.2.3 in \cite{Pe} for a proof).

These results can be used to deduce that $G$ admits a factorization of the form
\begin{equation}\label{badG}
			W^{*}\left(	\begin{array}{cc}
												u			&	\Zero\\
												\Zero	&	\Psi
									\end{array}\right) V^{*},
\end{equation}
where $u=\bar{z}\bar{\theta}\bar{h}/h$, $h$ is an outer function in $H^{2}$, $\theta$ is an inner function, $V=(\,v\;\bar{\Theta}\,)$
and $W^{t}=(\,w\;\bar{\Xi}\,)$ are thematic, and $\Psi\in L^{\infty}(\M_{m-1,n-1})$ satisfies $\|\Psi\|_{L^{\infty}(\M_{m-1,n-1})}\leq 1$.  
Conversely, it is easy to verify that any matrix function which admits a factorization of this form is badly approximable.

For the same matrix function $G$, it can also be shown that the Toeplitz operator $T_{u}$ is Fredholm with positive 
index, $\|H_{\Psi}\|_{\rm e}\leq\|H_{G}\|_{\rm e}$, and the matrix functions $\Theta$ and $\Xi$ are left-invertible in $H^{\infty}$,
i.e. there are matrix functions $A$ and $B$ in $H^{\infty}$ such that $A\Theta=I_{n-1}$ and $B\Xi=I_{m-1}$ hold.  

We refer the reader to Chapter 2 and Chapter 14 of \cite{Pe} for proofs of the previously mentioned results.

\begin{lem}\label{MainLemma}
			Suppose that $G\in L^{\infty}(\M_{m,n})$ is a matrix function of the form
			\begin{equation*}
						G=W^{*}\left(	\begin{array}{cc}
																u	&	\Zero\\
																\Zero	&	\Psi
															\end{array}\right) V^{*},
			\end{equation*}
			where $u$ is a unimodular function such that the Toeplitz operator $T_{u}$ is Fredholm with $\ind T_{u}\geq 0$, 
			$\Psi\in L^{\infty}(\mathbb{M}_{m-1,n-1})$ satisfies $\|\Psi\|_{L^{\infty}(\M_{m-1,n-1})}\leq 1$,  the matrix functions 
			$V=(\,v\;\bar{\Theta}\,)$ and $W^{t}=(\,w\;\bar{\Xi}\,)$ are thematic, and the bounded analytic 
			matrix functions $\Theta$ and $\Xi$ are left-invertible in $H^{\infty}$. Let $A$ and $B$ be left-inverses for $\Theta$ 
			and $\Xi$ in $H^{\infty}$, respectively, and $\xi\in\ker T_{\Psi}$.
			\begin{itemize}
						\item[1.]	If $\xi$ is co-outer, then $A^{t}\xi+a v$ is co-outer for any $a\in H^{2}$.
						\item[2.]	For $a\in H^{2}$, $A^{t}\xi+a v$ belongs to $\ker T_{G}$ if and only if $a$ satisfies
											\begin{equation}\label{PYequation}
														T_{u}a=\pP_{+}(w^{t}B^{*}\Psi\xi-u v^{*}A^{t}\xi).
											\end{equation}
			\end{itemize}
			Moreover, if $\xi_{\#}\Mydef A^{t}\xi+a v$ with $a\in H^{2}$ satisfying $(\ref{PYequation})$, then
			\begin{itemize}
						\item[3.]	$\eta_{\#}\Mydef\bar{z}\bar{G}\bar{\xi}_{\#}$ is co-outer whenever $\bar{z}\overline{\Psi \xi}$ 
											is co-outer, and
						\item[4.]	$\xi_{\#}\in\cM_{G}$ whenever $\xi\in\cM_{\Psi}$ and $\|H_{\Psi}\|=1$.
			\end{itemize}
\end{lem}
\begin{proof}
Notice that for any $a\in H^{2}$,
\begin{equation}\label{xiThetaEq}
			\Theta^{t}\xi_{\#}=\Theta^{t}(A^{t}\xi+av)=\xi,
\end{equation}
because $A$ is a left-inverse for $\Theta$ and $V$ is unitary-valued.  In particular, if the entries of $\xi$ do not have a 
common inner divisor, then the entries of $\xi^\#$ do not have a common inner divisor either.  This establishes assertion 1.

Although assertion 2 is contained in \cite{PY1} and \cite{PY2}, we provide a proof for future reference. 
Let $\xi_{\#}=A^{t}\xi+a v$. It follows from $(\ref{xiThetaEq})$ that 
\begin{equation}\label{VXiEq}
			V^{*}\xi_{\#}=\left(
										\begin{array}{c}
													v^{*}\xi_{\#}\\
													\Theta^{t}\xi_{\#}
										\end{array}\right)=\left(
										\begin{array}{c}
													v^{*}\xi_{\#}\\
													\xi
										\end{array}\right),
\end{equation}
and so $G\xi_{\#}=\bar{w}uv^{*}\xi_{\#}+\Xi\Psi\xi$.    Since $W$ is unitary-valued, then $I_{m}=\bar{w}w^{t}+\Xi\Xi^{*}$ 
holds and so
\[	B^{*}=I_{m}B^{*}=\bar{w}w^{t}B^{*}+\Xi\Xi^{*}B^{*}=\bar{w}w^{t}B^{*}+\Xi.	\]
In particular, $\Xi=(I_{n}-\bar{w}w^{t})B^{*}$ and so
\begin{equation}\label{GXiEq}
			G\xi_{\#}=\bar{w}u(v^{*}A^{t}\xi+a)+\Xi\Psi\xi=B^{*}\Psi\xi+\bar{w}(u(v^{*}A^{t}\xi+a)-w^{t}B^{*}\Psi\xi).
\end{equation}

In order to obtain the desired conclusion from ($\ref{GXiEq}$), we need the following well-known fact whose proof can 
be found in Chapter 14 of \cite{Pe}.

\begin{fact}\label{coOuterFact}
			Let $\Upsilon$ be a co-outer matrix function in $H^{2}(\M_{m,n})$.  If $\xi\in L^{2}(\C^{n})$ is such that
			$\Upsilon\xi\in H^{2}(\C^{m})$, then $\xi\in H^{2}(\C^{n})$.
\end{fact}

It follows now, from Fact \ref{coOuterFact}, that $G\xi_{\#}$ belongs to $H^{2}_{-}(\C^{m})$ if and only if 
$\pP_{+}(u(v^{*}A^{t}\xi+a)-w^{t}B^{*}\Psi\xi)=\Zero$ because $\Psi\xi\in H^{2}_{-}(\C^{m-1})$ and $w$ is co-outer.  
Thus, $G\xi_{\#}\in H^{2}_{-}(\C^{m})$ if and only if $T_{u}a=\pP_{+}(w^{t}B^{*}\Psi\xi-u v^{*}A^{t}\xi)$.  This 
completes the proof of 2.

Henceforth, we fix a function $a_{0}\in H^{2}$ that satisfies $(\ref{PYequation})$.  The existence of $a_{0}$ follows 
from the fact that $T_{u}$ is surjective.

To prove 3, observe that $(\ref{GXiEq})$ can be rewritten as
\[	G\xi_{\#}=B^{*}\Psi\xi+\bar{w}\bar{z}\bar{b}_{0}	\]
for some $b_{0}\in H^{2}$ because $\pP_{+}(u(v^{*}A^{t}\xi+a_{0})-w^{t}B^{*}\Psi\xi)=\Zero$. Let $\eta\Mydef\bar{z}\overline{\Psi\xi}$.  
Then $\eta_{\#}=\bar{z}\bar{G}\bar{\xi}^{\#}=B^{t}\eta+b_{0}w$ and so \[	\Xi^{t}\eta_{\#}=\Xi^{t}B^{t}\eta+b_{0}\Xi^{t}w=\eta,	\] 
because $B$ is a left-inverse of $\Xi$ and $W$ is unitary-valued.  Hence, $\eta_{\#}$ is co-outer whenever $\eta$ is co-outer.

Finally, we prove 4.  Since $\xi$ is a maximizing vector of $H_{\Psi}$ and belongs to $\ker T_{\Psi}$, then
$\|\Psi\xi\|_{2}=\|H_{\Psi}\xi\|_{2}=\|\xi\|_{2}$, as $\|H_{\Psi}\|=1$.  Moreover, since $H_{G}\xi_{\#}=G\xi_{\#}$, 
$W$ is unitary-valued, and
\begin{equation*}
			W G\xi_{\#}=\left(
										\begin{array}{c}
													u v^{*}\xi_{\#}\\
													\Psi\xi
										\end{array}\right),
\end{equation*}
we may conclude that
\begin{equation*}
			\|H_{G}\xi_{\#}\|_{2}^{2}	=\|W G\xi_{\#}\|_{2}^{2}=\|u v^{*}\xi_{\#}\|_{2}^{2}+\|\Psi\xi\|_{2}^{2}
																=\|v^{*}\xi_{\#}\|_{2}^{2}+\|\xi\|_{2}^{2}=\|\xi_{\#}\|_{2}^{2}
\end{equation*}
because $(\ref{VXiEq})$ holds and $V$ is unitary-valued.  Thus $\xi_{\#}\in\cM_{G}$.
\end{proof}

We are now ready to state and prove the main result of this section.

\begin{thm}\label{MainThm}
			Let $m,n \geq2$ and $G\in L^{\infty}(\M_{m,n})$ be a matrix function of the form
			\begin{equation*}
						G=W_{0}^{*}\left(	\begin{array}{cc}
																		u_{0}	&	\Zero\\
																		\Zero	&	\Psi_{0}
															\end{array}\right) V_{0}^{*},
			\end{equation*}
			where $u_{0}$ is a unimodular function such that the Toeplitz operator $T_{u_{0}}$ is Fredholm with $\ind T_{u_{0}}>0$, 
			$\Psi_{0}\in L^{\infty}(\M_{m-1,n-1})$, the matrix functions $V_{0}=(\,v_{0}\;\bar{\Theta}\,)$ and $W_{0}^{t}=(\,w_{0}\;\bar{\Xi}\,)$ 
			are thematic, and the bounded analytic matrix functions $\Theta$ and $\Xi$ are left-invertible in $H^{\infty}$.  Suppose that
			\begin{equation}\label{PsiFact}
						\Psi_{0}=W_{1}^{*}\left(	\begin{array}{cc}
																						u_{1}	&	\Zero\\
																						\Zero	&	\Psi_{1}
																			\end{array}\right) V_{1}^{*}
			\end{equation}
			for some unimodular function $u_{1}$ such that the Toeplitz operator $T_{u_{1}}$ is Fredholm with $\ind T_{u_{1}}>0$,
			 $\Psi_{1}\in L^{\infty}(\mathbb{M}_{m-2,n-2})$ such that $\|\Psi_{1}\|_{L^{\infty}(\mathbb{M}_{m-2,n-2})}\leq 1$,
			and thematic matrix functions $V_{1}$ and $W_{1}^{t}$.  Then $G$ admits a factorization of the form
			\begin{equation*}
						G=\cW^{*}\left(
							\begin{array}{cc}
										u			&	\Zero\\
										\Zero	&	\Delta
							\end{array}\right)\cV^{*}
			\end{equation*}
			for some unimodular function $u$ such that $T_{u}$ is Fredholm with index equal to $1$, a badly approximable matrix function 
			$\Delta$ such that $\|\Delta\|_{L^{\infty}(\M_{m-1,n-1})}=1$, and thematic matrix functions $\cV$ and $\cW^{t}$.
\end{thm}
\begin{proof}
			Let $A$ and $B$ be left-inverses of $\Theta$ and $\Xi$ in $H^{\infty}$, respectively, and $k_{j}\Mydef\ind T_{u_{j}}$ for $j=0,1$.  
			By Corollary \ref{badFact}, there is an outer function $h_{j}$ such that
			\[	u_{j}=\bar{z}^{k_{j}}\frac{\bar{h}_{j}}{h_{j}}	\] 
			and both $h_{j}$ and $1/h_{j}$ belong to $H^{p}$ for some $2<p\leq\infty$, for $j=0,1$.  Let $v_{1}$ denote the first column of 
			$V_{1}$ and $\xi\Mydef z^{k_{1}-1} h_{1}v_{1}$.  It follows at once from $(\ref{PsiFact})$ that $\Psi_{0}\xi=\bar{z}\bar{h}_{1}\bar{w}_{1}$.  
			Thus, $\xi$ is a maximizing vector of $H_{\Psi_{0}}$ and belongs to $\ker T_{\Psi_{0}}$.  In particular, the column function
			$\eta\Mydef\bar{z}\bar{\Psi}_{0}\bar{\xi}=h_{1}w_{1}$ is co-outer.
			
			Consider the equation
			\begin{equation}\label{keyEquation}
						T_{u_{0}}a=\pP_{+}(w_{0}^{t}B^{*}\Psi_{0}\xi-u_{0}v_{0}^{*}A^{t}\xi),\,a\in H^{2}.
			\end{equation}
			It follows from the surjectivity of the Toeplitz operator $T_{u_{0}}$ that there is an $a_{0}\in H^{2}$ that satisfies 
			$(\ref{keyEquation})$.  Furthermore, we may assume without loss of generality that $z$ is not an inner divisor $a_{0}$;
			otherwise, we consider $a_{0}+h_{0}$ instead of $a_{0}$.
			
			By Lemma \ref{MainLemma}, the column function 
			\[	\xi_{\#}\Mydef A^{t}\xi+a_{0}v_{0}	\] 
			is a maximizing vector of the Hankel operator $H_{G}$ and belongs to $\ker T_{G}$, as $\xi$ is a maximizing vector of the 
			Hankel operator $H_{\Psi_{0}}$ and $\|H_{\Psi_{0}}\|=1$.  Since $\Theta^{t}\xi_{\#}=\xi$ and $h_{1}v_{1}$ is co-outer, then 
			the greatest common inner divisor of the entries of $\xi_{\#}$ must be an inner divisor of $z^{k_{1}-1}$ by Fact \ref{coOuterFact}.  
			Therefore, $\xi_{\#}$ is co-outer whenever $z$ is not an inner divisor of $\xi_{\#}$.  On the other hand, $z$ is an inner divisor 
			of the entries of $\xi_{\#}$	if and only if $z$ is an inner divisor of $a_{0}$.  Since $z$ is not an inner divisor of $a_{0}$, 
			it follows that $\xi_{\#}$ is co-outer.
			
			From $(\ref{GXiEq})$ and $(\ref{keyEquation})$,
			\[	G\xi_{\#}=B^{*}\Psi_{0}\xi+\bar{w}_{0}\bar{z}\bar{b}_{0},	\]
			for some $b_{0}\in H^{2}$.  Thus the function
			\[	\eta_{\#}\Mydef\bar{z}\bar{G}\bar{\xi}_{\#}=B^{t}\eta+b_{0}w_{0}	\]  
			is co-outer as  well, by Lemma \ref{MainLemma}.
						
			From the remarks following Definition \ref{mDef}, we deduce that
			\[	\|\eta_{\#}(\zeta)\|_{\C^{m}}	=\|G\xi_{\#}(\zeta)\|_{\C^{m}}=\|G(\zeta)\|_{\M_{m,n}}\|\xi_{\#}(\zeta)\|_{\C^{n}}
																				=\|\xi_{\#}(\zeta)\|_{\C^{n}}	\]
			for a.e. $\zeta\in\T$ because $\xi_{\#}$ is a maximizing vector of the Hankel operator $H_{G}$ and belongs to $\ker T_{G}$.
			Let $h\in H^{2}$ be an outer function such that $|h(\zeta)|=\|\xi_{\#}(\zeta)\|_{\C^{n}}$ for a.e. $\zeta\in\T$.
			We obtain that
			\begin{equation}\label{etaEq}
						\|\eta_{\#}(\zeta)\|_{\C^{m}}	=\|\xi_{\#}(\zeta)\|_{\C^{n}}=|h(\zeta)|
			\end{equation}
			for a.e. $\zeta\in\T$ and so the column functions
			\[	\nu\Mydef \frac{1}{h}\xi_{\#}\, \mbox{ and }\;\omega\Mydef \frac{1}{h}\eta_{\#}	\] 
			are both inner and co-outer.
			
			Consider the unimodular function $u\Mydef\omega^{t}G\nu$.  It is easy to verify that
			\begin{equation}\label{factU}
						u=\frac{1}{h^{2}}(h\omega)^{t}G(h\nu)=\frac{1}{h^{2}}\bar{z}|h|^{2}=\bar{z}\frac{\bar{h}}{h},
			\end{equation}
			by $(\ref{etaEq})$, and
			\[	\|H_{u}\|_{\rm e}=\dist_{L^{\infty}}(u, H^{\infty}+C)\leq\dist_{L^{\infty}(\M_{m,n})} (G, (H^{\infty}+C)(\M_{m,n}))<1,	\] 
			because $\nu$ and $\omega$ are inner and $\|H_{G}\|_{\rm e}<1$.  Since $u$ satisfies $(\ref{factU})$ and $\|H_{u}\|_{\rm e}<1$, 
			it follows that $u$ is an admissible badly approximable scalar function, and so the Toeplitz operator $T_{u}$ is Fredholm with 
			positive index (see Section \ref{intro}) and therefore $T_{\bar{h}/h}$ is Fredholm.  Since $V_{0}$ is unitary-valued and 
			\begin{equation*}
						V_{0}^{*}\xi_{\#}=\left(
						\begin{array}{c}
									v^{*}_{0}\xi_{\#}\\
									\xi
						\end{array}\right),
			\end{equation*}
			then
			\begin{equation*}
						|h(\zeta)|^{2}=\|\xi_{\#}(\zeta)\|_{\C^{n}}^{2}=\|V_{0}^{*}\xi_{\#}(\zeta)\|_{\C^{n}}^{2}
														=|(v^{*}_{0}\xi_{\#})(\zeta)|^{2}+\|\xi(\zeta)\|_{\C^{n-1}}^{2}\geq|h_{1}(\zeta)|^{2}
			\end{equation*}
			holds for a.e. $\zeta\in\T$ and so $1/h\in H^{p}$.  By Theorem \ref{invThm}, $T_{\bar{h}/h}$ is invertible and so $\ind T_{u}=1$.
			
			Let $\cV$ and $\cW^{t}$ be thematic matrix functions whose first columns are $\nu$ and $\omega$, respectively.  (The existence
			of such matrix functions was mentioned in Section \ref{thematicSection}.) 	Since $\omega^{t}G\nu=u$ is unimodular, it follows that
			\begin{equation*}
						\cW G\cV=\left(
							\begin{array}{cc}
										u			&	\Zero\\
										\Zero	&	\Delta
							\end{array}\right)
			\end{equation*}
			for some bounded matrix function $\Delta\in L^{\infty}(\M_{m-1,n-1})$ with $L^{\infty}$-norm equal to 1, which is necessarily
			badly approximable.  This completes the proof.
\end{proof}

\begin{cor}\label{MainCor}
			Suppose that $G$ satisfies the hypothesis of Theorem {\em \ref{MainThm}}.  If $k$ is an integer satisfying $1\leq k\leq\ind T_{u_{1}}$, 
			then $G$ admits a factorization of the form
			\begin{equation*}
						G=\cW^{*}\left(
							\begin{array}{cc}
										u			&	\Zero\\
										\Zero	&	\Delta
							\end{array}\right)\cV^{*}
			\end{equation*}
			for some unimodular function $u$ such that $T_{u}$ is Fredholm with index equal to $k$, a badly approximable matrix function 
			$\Delta$ such that $\|\Delta\|_{L^{\infty}(\M_{m-1,n-1})}=1$, and thematic matrix functions $\cV$ and $\cW^{t}$.
\end{cor}
\begin{proof}
			Let $k$ be a fixed positive integer satisfying $k\leq\ind T_{u_{1}}$.  By Theorem \ref{MainThm}, the matrix function $z^{k-1}G$ admits
			a factorization of the form
			\begin{equation*}
						z^{k-1}G=\cW^{*}\left(
											\begin{array}{cc}
														u			&	\Zero\\
														\Zero	&	\Delta
											\end{array}\right)\cV^{*},
			\end{equation*}
			where $\ind T_{u}=1$, and so
			\begin{equation*}
						G=\cW^{*}\left(
										\begin{array}{cc}
													\bar{z}^{k-1}u	&	\Zero\\
													\Zero						&	\bar{z}^{k-1}\Delta
										\end{array}\right)\cV^{*}
			\end{equation*}
			is the desired factorization.
\end{proof}

At this point, we are unsatisfied with the conclusion of Corollary \ref{MainCor}.  After all, it does not give any information
concerning the matrix function $\Delta$.  Therefore, we ask, under some reasonable assumptions, whether the ``largest'' possible 
thematic index appearing as the first thematic index in a thematic factorization of $\Delta$ should equal $\ind T_{u_{0}}$. 
An affirmative answer is given in Theorem \ref{CompareIotas}.  Prior to stating and proving Theorem \ref{CompareIotas}, we  
introduce notation and recall some needed facts.
 
Suppose that $G$ is a badly approximable matrix function in $L^{\infty}(\M_{m,n})$ such that $\|H_{G}\|_{\rm e}<1$ and $\|H_{G}\|=1$.  
As mentioned in the remarks following Definition \ref{mDef}, $G$ admits a representation of the form $(\ref{badG})$ for some unimodular 
function $u$ such that the Toeplitz operator $T_{u}$ is Fredholm with $\ind T_{u}>0$.  It turns out that there is an upper bound on 
the possible values of the index of $T_{u}$ given by
\begin{equation}\label{iotaDef}
			\iota(H_{G})\Mydef\min\{	j>0: \|H_{z^{j}G}\|<\|H_{G}\|\}.
\end{equation}
Note that $\iota(H_{G})$ is a well-defined non-zero positive integer and depends only on the Hankel operator $H_{G}$ (and not on the 
choice of its symbol).  Moreover, there exists a (possibly distinct) factorization of $G$ of the form $(\ref{badG})$  such that 
$\ind T_{u}=\iota(H_{G})$ and $\iota(H_{\Psi})\leq\iota(H_{G})$.  See \cite{AP2} or Section 10 in Chapter 14 of \cite{Pe} 
for proofs of these facts.

\begin{thm}\label{CompareIotas}
			Let $m,n\geq 2$.  Suppose that $G\in L^{\infty}(\M_{m,n})$, $\|H_{G}\|_{\rm e}<1$, and $G$ admits the factorizations
			\begin{equation}\label{twoFact}
						G=W^{*}_{0}\left(
							\begin{array}{cc}
										u_{0}	&	\Zero\\
										\Zero	&	\Psi
							\end{array}\right)V^{*}_{0}
						=\cW^{*}\left(
							\begin{array}{cc}
										u			&	\Zero\\
										\Zero	&	\Delta
							\end{array}\right)\cV^{*},
			\end{equation}
			where $u_{0}$ and $u$ are unimodular functions such that the Toeplitz operators $T_{u_{0}}$ and $T_{u}$ are Fredholm with
			positive indices, $\Psi$ and $\Delta$ are badly approximable functions with $L^{\infty}$-norm equal to 1, and the matrix 
			function $V_{0}$, $W_{0}^{t}$, $\cV$, and $\cW^{t}$ are thematic.  If $\ind T_{u_{0}}=\iota(H_{G})$, 
			$\iota(H_{\Psi})<\iota(H_{G})$ and $\ind T_{u}\leq\iota(H_{\Psi})$, then 
			\begin{equation}\label{iotaIneq}
						\iota(H_{\Delta})\geq\iota(H_{G}).
			\end{equation}
			In addition, if $\ind T_{u}=\iota(H_{\Psi})$, then equality holds in $(\ref{iotaIneq})$.
\end{thm}
\begin{proof}
			Let $\iota\Mydef\iota(H_{G})$.  If $\iota(H_{\Delta})<\iota$, then
			\[	\|H_{z^{\iota-1}\Delta}\|<\|H_{\Delta}\|=1\,\mbox{ and }\;\ind T_{z^{\iota-1}u}\leq\iota(H_{\Psi})-(\iota-1)\leq 0.	\]
			It follows that the matrix function
			\begin{equation*}
						z^{\iota-1}G=\cW^{*}\left(
												\begin{array}{cc}
															z^{\iota-1}u	&	\Zero\\
															\Zero					&	z^{\iota-1}\Delta
												\end{array}\right)\cV^{*}
			\end{equation*}
			satisfies $\|H_{z^{\iota-1}G}\|<1=\|H_{G}\|$, by Lemma 14.10.7 in \cite{Pe}, and so $\iota(H_{G})\leq\iota-1$, a contradiction.
			This establishes $(\ref{iotaIneq})$.
			
			Suppose that $\ind T_{u}=\iota(H_{\Psi})$.  Let $j\Mydef\iota(H_{\Psi})$ and consider the factorizations
			\begin{equation*}
						z^{j}G=W_{0}^{*}\left(
							\begin{array}{cc}
										z^{j}u_{0}	&	\Zero\\
										\Zero				&	z^{j}\Psi
							\end{array}\right)V^{*}_{0}
						=\cW^{*}\left(
							\begin{array}{cc}
										z^{j}u	&	\Zero\\
										\Zero		&	z^{j}\Delta
							\end{array}\right)\cV^{*}.
			\end{equation*}	
			It is easy to see that the sum of the thematic indices of $z^{j}G$ corresponding to the superoptimal singular value 1 
			equals $\ind T_{z^{j}u_{0}}=\iota(H_{G})-\iota(H_{\Psi})$, because $\|H_{z^{j}\Psi}\|<\|H_{\Psi}\|=1$.  
			
			In order to proceed, we need the following lemma.
			
			\begin{lem}\label{MaxVectors}
						Let $G\in L^{\infty}(\M_{m,n})$ be such that $\|H_{G}\|_{\rm e}<1$ and $\|H_{G}\|=1$.  Suppose that $G$ is a badly 
						approximable matrix function that admits a representation of the form $(\ref{badG})$, in which $V$ and $W^{t}$ are 
						thematic matrix functions, $u$ is a unimodular function, and $\Psi$ is a bounded matrix function.  
						Let $V=(\,v\;\bar{\Theta}\,)$.
						\begin{enumerate}
									\item	If $f\in\cM_{G}$ satisfies $\Theta^{t} f=\Zero$, then $f=\xi v$ for some 
												$\xi\in\ker T_{u}$.
									\item	If $\Psi$ is a badly approximable matrix function with $L^{\infty}$-norm equal to 1 and the Toeplitz 
												operator $T_{u}$ is Fredholm with $\ind T_{u}\leq 0$, then 
												\begin{equation}\label{compareMaxVectors}
															\dim\cM_{G}\leq\dim\cM_{\Psi}.
												\end{equation}
												Moreover, if $\ind T_{u}=0$, then equality holds in $(\ref{compareMaxVectors})$.
						\end{enumerate}
			\end{lem}

			We finish the proof of Theorem \ref{CompareIotas} before proving Lemma \ref{MaxVectors}.
			
			As already seen,
			\[	\iota(H_{\Psi})<\iota(H_{G})\leq\iota(H_{\Delta})	\]
			and so $z^{j}\Delta$ is a badly approximable matrix function of $L^{\infty}$-norm equal to 1, since $\|H_{z^{j}\Delta}\|=
			\|H_{\Delta}\|=1$.  Let $\ell\Mydef\iota(H_{z^{j}\Delta})$.  Then $\|H_{z^{\ell}z^{j}\Delta}\|<\|H_{z^{j}\Delta}\|=1$ 
			implies that $\ell+j\geq\iota(H_{\Delta})$, and therefore
			\[	\dim\cM_{z^{j}G}=\dim\cM_{z^{j}\Delta}\geq\iota(H_{z^{j}\Delta})\geq\iota(H_{\Delta})-j=\iota(H_{\Delta})-\iota(H_{\Psi}),	\]
			by Lemma \ref{MaxVectors}.  Hence
			\[	\iota(H_{G})\geq\iota(H_{\Delta}),	\]
			because the sum of the thematic indices of $z^{j}G$ corresponding to the superoptimal singular value 1, namely 
			$\iota(H_{G})-\iota(H_{\Psi})$, equals $\dim\cM_{z^{j}G}$ (e.g. see Theorem 14.7.4 of \cite{Pe}).  This completes the proof.
\end{proof}

\begin{rmk}
			Note that if the inequality $\ind T_{u}\leq\iota(H_{\Psi})$ in Theorem \ref{CompareIotas} is strict, then equality in $(\ref{iotaIneq})$
			may not hold.  For instance, consider a monotone non-increasing thematic factorization (e.g. see Section \ref{SectionMonoFactorization}) 
			of any admissible unitary-valued very badly approximable matrix function $G\in L^{\infty}(\M_{2})$ with thematic indices $3$ and $2$,
			and any other thematic factorization of $G$ whose first thematic index equals 1.
\end{rmk}

\begin{proof}[Proof of Lemma \ref{MaxVectors}]	
Let $W^{t}=(\,w\;\bar{\Xi}\,)$.  To prove assertion 1, we may assume that $f\in\cM_{G}$ is non-zero.  Since
\begin{equation*}
			f=VV^{*}f=(\,v\;\bar{\Theta}\,)\left(	\begin{array}{c}
																									v^{*}f\\
																									\Zero
																						\end{array}\right)=v(v^{*}f),
\end{equation*}
Fact \ref{coOuterFact} implies $\xi\Mydef v^{*}f\in H^{2}$, as $v$ is co-outer.  It remains to show that $u\xi\in H^{2}_{-}$. 
Since  $u\xi\bar{w}=Gf\in H^{2}_{-}(\C^{m})$, it follows again from Fact \ref{coOuterFact} that $u\xi\in H^{2}_{-}$ because 
$w$ is co-outer.  Thus, $f=\xi v$ with $\xi\in\ker T_{u}$.

Suppose now that the functions $\Psi$ and $u$ satisfy the assumptions of assertion 2.  Let $\{f_{j}\}_{j=1}^{N}$ be a 
basis for $\cM_{G}$ and define $g_{j}=\Theta^{t}f_{j}$ for $1\leq j\leq N$.  Since $\ind T_{u}\leq 0$, then $\ker T_{u}$ is 
trivial, and each $g_{j}$ is a non-zero function in $H^{2}(\C^{n-1})$ by assertion 1.  Furthermore, $\{g_{j}\}_{j=1}^{N}$ 
is a linearly independent set in $H^{2}(\C^{n-1})$; after all, if there are scalars $c_{1},\ldots,c_{n}$ such that
\[	\Zero=\sum_{j=1}^{N}c_{j}g_{j}=\Theta^{t}\left(\sum_{j=1}^{N}c_{j}f_{j}\right),	\]
then $\sum_{j=1}^{N}c_{j}f_{j}=\Zero$ by assertion 1, and so $c_{j}=0$ for $1\leq j\leq N$ because 
$\{f_{j}\}_{j=1}^{N}$ is a linearly independent set.  In order to prove $(\ref{compareMaxVectors})$, it suffices 
to show that $g_{j}$ belongs to $\cM_{\Psi}$ for $1\leq j\leq N$.  To this end, fix $j_{0}$ such that $1\leq j_{0}\leq N$. 
Since $G$ is badly approximable and admits a factorization of the form $(\ref{badG})$, then 
\[	\|f_{j_{0}}(\zeta)\|^{2}_{\C^{n}}=\|Gf_{j_{0}}(\zeta)\|^{2}_{\C^{m}}=|v^{*}f_{j_{0}}(\zeta)|^{2}
																			+\|\Psi\Theta^{t}f_{j_{0}}(\zeta)\|^{2}_{\C^{m-1}}	\]
for a.e. $\zeta\in\T$.  On the other hand,
\[	|v^{*}f_{j_{0}}(\zeta)|^{2}+\|\Theta^{t}f_{j_{0}}(\zeta)\|^{2}_{\C^{n-1}}=\|V^{*}f_{j_{0}}(\zeta)\|^{2}_{\C^{n}}
																																							=\|f_{j_{0}}(\zeta)\|^{2}_{\C^{n}}	\]
holds for a.e. $\zeta\in\T$ because $V$ is unitary-valued.  Thus, the function $g_{j_{0}}=\Theta^{t}f_{j_{0}}$ satisfies
\[	\|\Psi g_{j_{0}}(\zeta)\|_{\C^{m-1}}=\|g_{j_{0}}(\zeta)\|_{\C^{n-1}} \mbox{ for a.e. }\zeta\in\T.	\]
Since $W$ is unitary-valued,
\begin{equation*}
			W Gf_{j_{0}}=\left(	\begin{array}{c}
												uv^{*}f_{j_{0}}\\
												\Psi\Theta^{t}f_{j_{0}}
									\end{array}\right),
\end{equation*}
and so $\Psi\Theta^{t} f_{j_{0}}=\Xi^{*}Gf_{j_{0}}\in H^{2}_{-}(\C^{m-1})$ by Fact \ref{coOuterFact} because 
$Gf_{j_{0}}\in H^{2}_{-}(\C^{m})$.  Hence, we may conclude that $g_{j_{0}}\in H^{2}(\C^{m-1})$ satisfies
\[	\|H_{\Psi}g_{j_{0}}\|_{2}=\|\Psi g_{j_{0}}\|_{2}=\|g_{j_{0}}\|_{2},	\]
i.e. each $g_{j_{0}}$ is a maximizing vector of the Hankel operator $H_{\Psi}$.  This completes the proof of 
$(\ref{compareMaxVectors})$.

Suppose now that $\ind T_{u}=0$.  Let $\xi_{1},\ldots,\xi_{d}$ be a basis for $\cM_{\Psi}$.  By Lemma \ref{MainLemma}, 
each function $\xi_{j}$ induces a function $\xi_{\#}^{(j)}=A^{t}\xi_{j}+a_{j}v\in\cM_{G}$ for some suitable $a_{j}\in H^{2}$,
$1\leq j\leq d$, where $A$ denotes a left-inverse of $\Theta$ in $H^{\infty}$.  It is easy to see from $(\ref{VXiEq})$ that 
$\xi_{\#}^{(1)},\ldots, \xi_{\#}^{(d)}$ form a linearly independent set, as $\{\xi_{j}\}_{j=1}^{d}$ is a linearly 
independent set.  Hence $\dim\cM_{G}=\dim\cM_{\Psi}$. 
\end{proof}

\begin{rmk}\label{exMaxVectors}
			Notice that the inequality given in \ref{compareMaxVectors} of Lemma \ref{MaxVectors} may in fact be strict.
			For instance, consider the badly approximable matrix function
			\begin{displaymath}
						G=\left(
							\begin{array}{cc}
										\bar{z}	&	\Zero\\
										\Zero		&	1
							\end{array}\right).
			\end{displaymath}
			It is easy to see that $\dim\cM_{G}=1$ and $G$ admits a factorization of the form
			\begin{equation*}
						G=W^{*}\left(
							\begin{array}{cc}
										z			&	\Zero\\
										\Zero	&	\bar{z}^{2}
							\end{array}\right)V^{*},
			\end{equation*}
			where
			\begin{equation*}
						W^{t}=\frac{1}{\sqrt{2}}\left(
							\begin{array}{cc}
										1	&	-\bar{z}\\
										z	&	1
							\end{array}\right)\,\mbox{ and }\;							
						V=\frac{1}{\sqrt{2}}\left(
							\begin{array}{cc}
										z^{2}	&	-1\\
										1			&	\bar{z}^{2}	
							\end{array}\right)
			\end{equation*}
			are thematic, and the Toeplitz operators $T_{z}$ and $T_{\bar{z}^{2}}$ are Fredholm with indices $-1$ and $2$, respectively.  
			By setting $\Psi\Mydef\bar{z}^{2}$, it is easy to see that $\dim\cM_{\Psi}=\dim\ker T_{\Psi}=\ind T_{\Psi}=2>\dim\cM_{G}$
			because $\Psi$ is unimodular.	
\end{rmk}

\section{Sequences of thematic indices}\label{SectionMonoFactorization}

We proceed by proving the existence of a monotone non-decreasing thematic factorization and show that other thematic 
factorizations are induced by a given monotone non-increasing thematic factorization.

\begin{definition}
			Let $G\in L^{\infty}(\M_{m,n})$ be a badly approximable matrix function whose superoptimal singular values $t_{j}=t_{j}(G)$,
			 $j\geq 0$, satisfy
			\begin{equation}\label{partialThematicCondition}
						\|H_{G}\|_{\rm e}<t_{r-1},\; t_{0}=\ldots=t_{r-1},\,\mbox{ and }\; t_{r-1}>t_{r}.
			\end{equation}
			We say that
			\[	(\,k_{0}, k_{1}, k_{2}, \ldots, k_{r-1}\,)	\]
			is \emph{a sequence of thematic indices for $G$} if $G$ admits a partial thematic factorization of the form 
			\begin{equation}\label{partialThem}
						W^{*}_{0}\cdot\ldots\cdot W^{*}_{r-1}\left(
						\begin{array}{ccccc}
									t_{0}u_{0}		&	\Zero				&	\ldots	&	\Zero					&	\Zero\\
									\Zero					&	t_{0}u_{1}	&	\ldots	&	\Zero					&	\Zero\\
									\vdots				&	\vdots			&	\ddots	&	\vdots				&	\vdots\\
									\Zero					&	\Zero				&	\ldots	&	t_{0}u_{r-1}	&	\Zero\\
									\Zero					&	\Zero				&	\ldots	&	\Zero					&	\Psi
						\end{array}\right)V^{*}_{r-1}\cdot\ldots\cdot V^{*}_{0},
			\end{equation}
			such that $\ind T_{u_{j}}=k_{j}$ and the matrix functions $V_{j}$ and $W_{j}$ are of the form $(\ref{almostThematic})$ 
			for $0\leq j\leq r-1$, and $\Psi$ satisfies $(\ref{PsiCondition})$. 
\end{definition}

\begin{thm}\label{NonDecSeq}
			Suppose that $G\in L^{\infty}(\M_{m,n})$ is a badly approximable matrix function satisfying $(\ref{partialThematicCondition})$.
			If $\nu$ equals the sum of the thematic indices corresponding to the superoptimal singular value $t_{0}(G)$, then
			\begin{equation}\label{oneSeq}	
						(\underbrace{1, 1,\ldots, 1}_{r-1}, \nu-r+1)
			\end{equation}
			is a sequence of thematic indices for $G$.  In particular, $G$ admits a monotone non-decreasing thematic factorization.
\end{thm}
\begin{proof}		
			Consider any thematic factorization of $\Delta_{0}\Mydef t_{0}^{-1}G$.  It follows from Theorem \ref{MainThm} that 
			$\Delta_{0}$ admits a factorization of the form
			\begin{equation*}
						\Delta_{0}=\cW_{0}^{*}\left(
								\begin{array}{cc}
											u_{0}	&	\Zero\\
											\Zero	&	\Delta_{1}
								\end{array}\right)\cV^{*}_{0}
			\end{equation*}
			where $\ind T_{u_{0}}=1$.  Similarly, Theorem \ref{MainThm} implies that $\Delta_{1}$ also admits a factorization 
			of the form
			\begin{equation*}
						\Delta_{1}=\cW_{1}^{*}\left(
								\begin{array}{cc}
											u_{1}	&	\Zero\\
											\Zero	&	\Delta_{2}
								\end{array}\right)\cV^{*}_{1}
			\end{equation*}
			where $\ind T_{u_{1}}=1$.  Continuing in this manner, we obtain matrix functions $\Delta_{0},\Delta_{1},\ldots,
			\Delta_{r-2},\Delta_{r-1}$ with factorizations of the form
			\begin{equation*}
						\Delta_{j}=\cW_{j}^{*}\left(
								\begin{array}{cc}
											u_{j}	&	\Zero\\
											\Zero	&	\Delta_{j+1}
								\end{array}\right)\cV^{*}_{j},
			\end{equation*}
			where $\ind T_{u_{j}}=1$, for $0\leq j\leq r-2$.  It is easy to see that these matrix functions induce a partial thematic
			factorization of $G$ in which the first $r-1$ thematic indices equal 1.  Since the sum of the thematic indices $\nu$
			of $G$ is independent of the partial thematic factorization, it must be that the $r$th thematic index in this induced
			partial thematic factorization equals $\nu-(r-1)$.
\end{proof}

The following corollaries are immediate.

\begin{cor}
			If $G\in L^{\infty}(\M_{m,n})$ is an admissible very badly approximable matrix function, then $G$ admits a monotone 
			non-decreasing thematic factorization.
\end{cor}

\begin{cor}
			If $G\in(H^{\infty}+C)(\M_{m,n})$ is a very badly approximable matrix function, then $G$ admits a monotone non-decreasing 
			thematic factorization.
\end{cor}

We go on to show that the thematic indices obtained in a monotone non-decreasing thematic factorization are \emph{not} 
uniquely determined.  Moreover, we determine all possible sequences of thematic indices for an admissible very badly 
approximable unitary-valued $2\times 2$ matrix function. 

\begin{thm}\label{unitarySeq}
			Let $U\in L^{\infty}(\M_{2})$ be an admissible very badly approximable unitary-valued matrix function.  Suppose 
			that $(k_{0},k_{1})$ is the monotone non-increasing sequence of thematic indices for $U$.  Then the collection
			of sequences of thematic indices for $U$ coincides with the set
			\[	\cS_{U}\Mydef\{(k_{1}-j,k_{0}+j): 0\leq j<k_{1}\}\cup\{(k_{0},k_{1})\}.	\]
\end{thm}
\begin{proof}
			Let $0\leq j<k_{1}$.  By Corollary \ref{MainCor}, $U$ admits a factorization of the form
			\begin{equation*}
						U=\cW^{*}\left(
								\begin{array}{cc}
											u_{0}	&	\Zero\\
											\Zero	&	u_{1}
								\end{array}\right)\cV^{*}
			\end{equation*}
			with $\ind T_{u_{0}}=k_{1}-j$.  Since the sum of the thematic indices of $U$ is independent of the thematic 
			factorization, it must be that $\ind T_{u_{1}}=k_{0}+j$.  Thus $\cS_{U}$ consists of sequences of thematic indices
			for $U$.
			
			Suppose now that $(a,b)$ is a sequence of thematic indices for $U$ that does not belong to $\cS_{U}$.  In this case, 
			$U$ admits a factorization of the form
			\begin{equation*}
						U=W^{*}\left(
								\begin{array}{cc}
											u_{0}	&	\Zero\\
											\Zero	&	u_{1}
								\end{array}\right)V^{*}
			\end{equation*}
			for some thematic matrix functions $V$ and $W^{t}$, and	unimodular functions $u_{0}$ and $u_{1}$ such that 
			$\ind T_{u_{0}}=a$ and $\ind T_{u_{1}}=b$.  Since $(a,b)\notin\cS_{U}$, it follows that $b>a$ and $a>k_{1}$.  
			Thus,	
			\begin{equation*}
						z^{k_{1}}U=W^{*}\left(
								\begin{array}{cc}
											z^{k_{1}}u_{0}	&	\Zero\\
											\Zero						&	z^{k_{1}}u_{1}
								\end{array}\right)V^{*}
			\end{equation*}
			is a very badly approximable unitary-valued matrix function.  In particular, $z^{k_{1}}U$ admits a
			monotone non-increasing thematic sequence, say $(\alpha,\beta)$.  Hence, $(\alpha+k_{1},\beta+k_{1})$
			is a monotone non-increasing sequence of thematic indices for $U$ and so, by the uniqueness of a monotone 
			non-increasing sequence, $k_{1}=\beta+k_{1}$ for some $\beta\geq1$ a contradiction.  This completes the proof.
\end{proof}

We now recall how monotone non-increasing thematic factorizations were obtained in \cite{AP2}.

Let $G\in L^{\infty}(\M_{m,n})$ be a badly approximable matrix function such that $(\ref{partialThematicCondition})$ holds.  
In this case, it is known that $G$ admits a monotone \emph{non-increasing} partial thematic factorization and that the thematic 
indices appearing in any monotone non-increasing partial thematic factorization of $G$ are uniquely determined by $G$.  In fact, 
as discussed in Section \ref{SectionBadlyApp}, $G_{0}=t_{0}^{-1}G$ admits a factorization of the form
\begin{equation*}
			G_{0}=W_{0}^{*}\left(
					\begin{array}{cc}
								u_{0}	&	\Zero\\
								\Zero	&	G_{1}
					\end{array}\right)V^{*}_{0}
\end{equation*}
with $\ind T_{u_{0}}=\iota(H_{G_{0}})$ and $\iota(H_{G_{0}})\geq\iota(H_{G_{1}})$ (see $(\ref{iotaDef})$).  Similarly, 
for each $1\leq j\leq r-1$, we obtain a matrix function $G_{j}$ with a factorization of the form
\begin{equation*}
			G_{j}=\breve{W}^{*}_{j}\left(
					\begin{array}{cc}
								u_{j}	&	\Zero\\
								\Zero	&	G_{j+1}
					\end{array}\right)\breve{V}^{*}_{j}
\end{equation*}
with $\ind T_{u_{j}}=\iota(H_{G_{j}})$ and $\iota(H_{G_{j}})\geq\iota(H_{G_{j+1}})$.  Then
\[	(\iota(H_{G_{0}}), \iota(H_{G_{1}}), \ldots, \iota(H_{G_{r-1}}))	\]
is the monotone non-increasing sequence of thematic indices for $G$.  (See \cite{AP2} or Section 10 in Chapter 14 of \cite{Pe}.)

Note that, in the general setting of $m\times n$ matrix functions, at least two sequences of thematic indices for $G$ exist; 
the monotone non-increasing sequence and the sequence in $(\ref{oneSeq})$.  The question remains: \emph{Are there any others?}

\begin{thm}\label{themSequence}
			Suppose $G\in L^{\infty}(\M_{m,n})$ is a badly approximable matrix function satisfying $(\ref{partialThematicCondition})$. If
			\[	(\,k_{0}, k_{1}, k_{2}, \ldots, k_{r-1}\,)	\]
			is the monotone non-increasing sequence of thematic indices for $G$, then
			\[	(\,k_{1}, k_{0}, k_{2}, \ldots, k_{r-1}\,)	\] 
			is also sequence of thematic indices for $G$.
\end{thm}
\begin{proof}
			Without loss of generality, we may assume that $t_{0}=1$ and $k_{0}>k_{1}$.  By Theorems \ref{MainThm} and \ref{CompareIotas}, 
			$G$ admits a thematic factorization of the form
			\begin{equation*}
						G=\cW^{*}\left(
							\begin{array}{cc}
										u			&	\Zero\\
										\Zero	&	\Delta
							\end{array}\right)\cV^{*},
			\end{equation*}
			where $\ind T_{u}=k_{1}$ and $\iota(H_{\Delta})=k_{0}$.  Let $(\,\kappa_{1}, \kappa_{2}, \ldots, \kappa_{r-1}\,)$ be
			the monotone non-increasing sequence of thematic indices for $\Delta$.  In particular, $\kappa_{1}=k_{0}$ and
			\begin{equation}\label{newSeq}
						(\,k_{1}, k_{0}, \kappa_{2}, \ldots, \kappa_{r-1}\,)
			\end{equation}
			is a sequence of thematic indices for $G$.  We claim that 
			\[	\kappa_{j}=k_{j}\,\mbox{ for }2\leq j\leq r-1.	\]
			
			By considering the monotone non-increasing sequence for $G$, it is easy to see that the sum of the thematic indices 
			corresponding to the superoptimal singular value 1 of $z^{k_{2}}G$ equals 
			\[	(k_{0}-k_{2})+(k_{1}-k_{2}).	\]
			On the other hand, this sum is also equal to 
			\[	(k_{1}-k_{2})+(k_{0}-k_{2})+\sum_{\{j\geq 2:\; \kappa_{j}\geq k_{2}\}}(\kappa_{j}-k_{2}),	\]
			because the sequence in $(\ref{newSeq})$ is a sequence of thematic indices for $G$.  This implies that $\kappa_{2}\leq k_{2}$.
			Now, by considing the matrix function $z^{\kappa_{2}}G$, the same argument reveals that $k_{2}\leq \kappa_{2}$.  
			Therefore $\kappa_{2}=k_{2}$.
			
			Let $2\leq\ell<r-1$.  Suppose we have already shown that $\kappa_{j}=k_{j}$ for $2\leq j\leq\ell$.
			In the same manner, the sum of the thematic indices corresponding to the superoptimal singular value 1
			of $z^{k_{\ell+1}}G$ equals 
			\[	\sum_{j=0}^{\ell}(k_{j}-k_{\ell+1})\,\mbox{ and }\;\sum_{j=0}^{\ell}(k_{j}-k_{l+1})+
					\sum_{\{j\geq\ell+1:\; \kappa_{j}\geq k_{l+1}\}}(\kappa_{j}-k_{l+1}).	\] 
			This implies that $\kappa_{\ell+1}\leq k_{l+1}$, and a similar argument shows $k_{l+1}\leq\kappa_{l+1}$.
			Hence we must have that $\kappa_{j}=k_{j}$ for $2\leq j\leq r-1$.
\end{proof}

Theorem \ref{themSequence} provides a stronger conclusion than one might think.  Loosely speaking, it says that we can 
always interchange the highest two adjacent thematic indices in any monotone non-increasing sequence of thematic indices 
and still obtain another sequence of thematic indices for the same matrix function.  Let us illustrate this with the 
following example.

\begin{ex}
			For simplicity, consider the very badly approximable function
			\begin{equation*}
						G=\left(\begin{array}{ccc}
													\bar{z}^{3}	&	\Zero				&	\Zero\\
													\Zero				&	\bar{z}^{2}	&	\Zero\\
													\Zero				&	\Zero				&	\bar{z}
										\end{array}
							\right).
			\end{equation*}
			Clearly, $(3, 2, 1)$ is the monotone non-increasing sequence of thematic indices for $G$.  Our results imply 
			that there are many other sequences of thematic indices for $G$.  Indeed, by considering the subsequence $(2, 1)$,
			Theorem \ref{themSequence} implies that $(\;3, 1, 2\;)$ is also a sequence of thematic indices for $G$.  Similarly, 
			it is easy to see that $(\;2, 3, 1\;)$ and $(\;2, 1, 3\;)$ are also sequences of thematic indices for $G$.  On the 
			other hand, it follows from Theorem \ref{NonDecSeq} that $(\;1, 1, 4\;)$ is a sequence of thematic indices for $G$.
			
			This leads us to ask:  \emph{Are there other sequences of thematic indices in which the first index is equal to 1?}
			
			It can be verified that $G$ admits the following thematic factorizations:
			\begin{align*}
						G&=					
						\frac{1}{\sqrt{2}}\left(
							\begin{array}{ccc}
										1				&	1			&	\Zero\\
										\Zero		&	\Zero	&	\sqrt{2}\\
										1				&	-1		&	\Zero
							\end{array}\right)\left(
							\begin{array}{ccc}
										\bar{z}	&	\Zero				&	\Zero\\
										\Zero		&	\bar{z}^{3}	&	\Zero\\
										\Zero		&	\Zero				&	\bar{z}^{2}
							\end{array}\right)
							\frac{1}{\sqrt{2}}\left(
							\begin{array}{ccc}
										\bar{z}^{2}	&	\Zero			&	1\\
										1						&	\Zero			&	-z^{2}\\
										\Zero				&	\sqrt{2}	&	\Zero
							\end{array}\right)\\
						&=						
						\frac{1}{\sqrt{2}}\left(
							\begin{array}{ccc}
										1					&	-z		&	\Zero\\
										\bar{z}		&	1			&	\Zero\\
										\Zero			&	\Zero	&	\sqrt{2}
							\end{array}\right)\left(
							\begin{array}{ccc}
										\bar{z}	&	\Zero				&	\Zero\\
										\Zero		&	\bar{z}^{4}	&	\Zero\\
										\Zero		&	\Zero				&	\bar{z}
							\end{array}\right)
							\frac{1}{\sqrt{2}}\left(
							\begin{array}{ccc}
										\bar{z}^{2}	&	1				&	\Zero\\
										-1					&	z^{2}		&	\Zero\\
										\Zero				&	\Zero		&	\sqrt{2}
							\end{array}\right).
			\end{align*}
			Thus,	$(\;1, 3, 2\;)$ and $(\;1, 4, 1\;)$	are sequences of thematic indices for $G$ as well.  These sequences induce 
			two others by considering the subsequences $(\;3, 2\;)$ and $(\;4, 1\;)$; namely $(\;1, 2, 3\;)$ and $(\;1, 1, 4\;)$.
			Thus, the matrix function $G$ admits at least 8 different sequences of thematic indices, namely
			\begin{align*}
						(\;3, 2, 1\;),\;(\;3, 1, 2\;),\;(\;2, 3, 1\;),&\;(\;2, 1, 3\;),\;(\;1, 3, 2\;),\\
						(\;1, 2, 3\;),\;(\;1, 4, 1\;),\,&\mbox{ and }(\;1, 1, 4\;).
			\end{align*}
			It is easy to verify that these are all possible sequences of thematic indices for $G$.
\end{ex}

\section{Unitary-valued very badly approximable $2\times 2$ matrix functions}\label{unitarySection}

The problem of finding all possible sequences of thematic indices for an arbitrary admissible very badly approximable 
matrix function seems rather difficult for $m\times n$ matrix functions with $\min\{m,n\}>2$.  However, in the case of 
unitary-valued $2\times 2$ matrix functions, the problem has a straightforward solution provided by Theorem \ref{unitarySeq}.  
In this section, we introduce a simple algorithm that yields thematic factorizations with desired thematic indices 
for such matrix functions.

\smallskip

\begin{center}
			\textbf{Algorithm}
\end{center}
Let $U$ be an admissible very badly approximable unitary-valued $2\times 2$ matrix function on $\T$ and $(k_{0},k_{1})$
denote the monotone non-increasing sequence of thematic indices for $U$.  Suppose $U$ admits a monotone non-increasing 
thematic factorization of the form
\begin{equation}\label{MyG}
			U=(\,\bar{w}_{0}\; \Xi\,)\left(
				\begin{array}{cc}
							\bar{z}^{k_{0}}\frac{\bar{h}_{0}}{h_{0}}	& \Zero\\
							\Zero																			&	\bar{z}^{k_{1}}\frac{\bar{h}_{1}}{h_{1}}
				\end{array}\right)\left(
				\begin{array}{c}
							v^{*}\\
							\Theta^{t}
				\end{array}\right),
\end{equation}
where $h_{0}$, $h_{1}$, and their respective inverses belong to $H^{p}$ for some $2<p\leq\infty$.  For each integer 
$j$ satisfying $1\leq j\leq k_{1}$, a thematic factorization of $U$ with thematic indices $(j, k_{0}+k_{1}-j)$ 
can be obtained as follows.
\begin{enumerate}
			\item	Find left-inverses $A$ and $B$ in $H^{\infty}$ for $\Theta$ and $\Xi$, respectively.
			\item	Set $\displaystyle{u_{0}=\bar{z}^{k_{0}-j+1}\frac{\bar{h}_{0}}{h_{0}}}$ and 
						$\displaystyle{\Psi=\bar{z}^{k_{1}-j+1}\frac{\bar{h}_{1}}{h_{1}}}$.
			\item	Let $\xi=z^{k_{1}-j}h_{1}$.  Find a solution $a_{0}\in H^{2}$ to the equation 
						\[	T_{u}a_{0}=\pP_{+}(w^{t}B^{*}\Psi-uv^{*}A^{t})\xi.	\]
						If $j<k_{1}$, we require, in addition, that $z$ is not an inner divisor of $a_{0}$.
						(Note that if $z$ is an inner divisor of $a_{0}$, then it suffices to replace $a_{0}$
						with $a_{0}+h_{0}$.)
			\item	Let $\xi_{\#}=A^{t}\xi+a_{0} v$ and $\eta_{\#}=\bar{z}\bar{G}\bar{\xi}_{\#}$.  Choose an outer 
						function $h\in H^{2}$ such that $|h(\zeta)|=\|\xi_{\#}(\zeta)\|_{\C^{2}}$ for a.e. $\zeta\in\T$.
			\item	Let $\nu=h^{-1}\xi_{\#}$ and $\omega=h^{-1}\eta_{\#}$.  Find thematic completions 
						\[	\cV=(\,\nu\;\bar{\Upsilon}\,)\,\mbox{ and }\;\cW^{t}=(\,\omega\;\bar{\Omega}\,)	\]
						to $\nu$ and $\omega$,  respectively. 
			\item	The desired thematic factorization for $G$ is given by
						\begin{equation}
									G=\cW^{*}\left(
									\begin{array}{cc}
												u			&	\Zero\\
												\Zero	&	\Delta
									\end{array}\right)\cV^{*}
						\end{equation}
						where
						\[	u=\bar{z}^{j}\frac{\bar{h}}{h}\quad\mbox{ and }\quad\Delta=\Omega^{*}G\bar{\Upsilon}.	\]
\end{enumerate}
\begin{center}
			\textbf{End of algorithm}
\end{center}

The validity of this algorithm is justified by the proof of Theorem \ref{MainThm} and Corollary \ref{MainCor}.

For matrix functions $G\in\cR(\M_{2})$, the badly approximable scalar functions appearing in the diagonal factor 
of $(\ref{MyG})$ also belong to $\cR$.  This is a consequence of the results in \cite{PY1} (see also Sections 5 and 
12 of Chapter 14 in \cite{Pe}).  As mentioned in Remark \ref{classR}, the outer functions $h_{0}$ and $h_{1}$ are 
(up to a multiplicative constant) products of quotients of reproducing kernels of $H^{2}$.  Therefore, steps 1 
through 6 of the algorithm are more easily implemented if $G\in\cR(\M_{2})$.

\begin{ex}\label{firstEx}
			Consider the matrix function
			\begin{equation}
						G=\frac{1}{\sqrt{2}}\left(
						\begin{array}{cc}
									\bar{z}^{4}	&	-\bar{z}^{2}\\
									\bar{z}^{3}	&	\bar{z}
						\end{array}\right)
						=\frac{1}{\sqrt{2}}\left(
						\begin{array}{cc}
									\bar{z}	&	-1\\
									1				&	z
						\end{array}\right)\left(
						\begin{array}{cc}
									\bar{z}^{3}	&	\Zero\\
									\Zero				&	\bar{z}^{2}
						\end{array}\right).
			\end{equation}
			Let			
			\begin{align}\label{ex1Functions}
						w=\frac{1}{\sqrt{2}}\left(
						\begin{array}{c}
									z\\
									1
						\end{array}\right),\;
						\Xi=\frac{1}{\sqrt{2}}\left(
						\begin{array}{c}
									-1\\
									z
						\end{array}\right),\;
						v=\left(
						\begin{array}{c}
									1\\
									\Zero
						\end{array}\right),\;&
						\Theta=\left(
						\begin{array}{c}
									\Zero\\
									1
						\end{array}\right),\nonumber\\
						h_{0}=h_{1}=1,\; B=\sqrt{2}\left(\,-1\quad\Zero\,\right),\,\mbox{ and }&\;A=\left(\,\Zero\quad1\,\right).
			\end{align}
			We find thematic factorizations with sequences of indices $(2,3)$ and $(1,4)$.
						
			\begin{enumerate}
						\item	\emph{A thematic factorization for $G$ with sequence of indices $(2,3)$:}	
						
									Let
									\[	u_{0}=\bar{z}^{2},\,\Psi=\bar{z},\,\xi=1,\,\mbox{ and }\;a_{0}=-z^{2}.	\]
									In this case, it is easy to verify that
									\begin{equation*}
												\xi_{\#}=\left(
												\begin{array}{c}
															-z^{2}\\
															1
												\end{array}\right)\,\mbox{ and }\;
												\eta_{\#}=\frac{1}{\sqrt{2}}\left(
												\begin{array}{c}
															-2\\
															\Zero
												\end{array}\right).
									\end{equation*}
									Since $\|\xi_{\#}(\zeta)\|_{\C^{2}}=2$ on $\T$, we may take $h(\zeta)=\sqrt{2}$ for $\zeta\in\T$.  Then	
									\begin{equation*}
												\nu=\frac{1}{\sqrt{2}}\left(
												\begin{array}{c}
															-z^{2}\\
															1
												\end{array}\right)\,\mbox{ and }\;
												\omega=\left(
												\begin{array}{c}
															-1\\
															\Zero
												\end{array}\right)
									\end{equation*}
									have thematic completions
									\[	\cV=(\,\nu\;\bar{\Upsilon}\,)\,\mbox{ and }\;\cW^{t}=(\,\omega\;\bar{\Omega}\,),	\]
									where
									\begin{equation*}
												\Upsilon=\frac{1}{\sqrt{2}}\left(
												\begin{array}{c}
															1\\
															z^{2}
												\end{array}\right)\,\mbox{ and }\;
												\Omega=\left(
												\begin{array}{c}
															\Zero\\
															1
												\end{array}\right).
									\end{equation*}
									Thus, $G$ admits the factorization
									\begin{equation*}
												G=\left(\begin{array}{cc}
																			-1		&	\Zero\\
																			\Zero	&	1
																\end{array}\right)
												\left(\begin{array}{cc}
																			\bar{z}^{2}	&	\Zero\\
																			\Zero				&	\bar{z}^{3}
																\end{array}\right)\frac{1}{\sqrt{2}}\left(
																\begin{array}{cc}
																			-\bar{z}^{2}	&	1\\
																			1						&	z^{2}
																\end{array}\right),
									\end{equation*}
									with sequence of thematic indices $(2,3)$ as desired.
									
						\item	\emph{A thematic factorization for $G$ with sequence of indices $(1,4)$:}	
						
									Let
									\[	u_{0}=\bar{z}^{3},\,\Psi=\bar{z}^{2},\,\xi=z,\,\mbox{ and }\;a_{0}=1-z^{3}.	\]
									It is easy to verify that
									\begin{equation*}
												\xi_{\#}=\left(
												\begin{array}{c}
															1-z^{3}\\
															z
												\end{array}\right)\;\mbox{ and }\;
												\eta_{\#}=\frac{1}{\sqrt{2}}\left(
												\begin{array}{c}
															z^{3}-2\\
															z^{2}
												\end{array}\right).
									\end{equation*}
									Since $\|\xi_{\#}(\zeta)\|_{\C^{2}}^{2}=3-\bar{z}^{3}-z^{3}$ on $\T$, we may choose 
									\[	h=a^{2}-\frac{1}{a^{2}}z^{3},\mbox{ where }a=\frac{-1+\sqrt{5}}{2}.	\] 
									By setting
									\begin{equation*}
												\nu=\frac{1}{h}\left(
												\begin{array}{c}
															1-z^{3}\\
															z
												\end{array}\right)\;\mbox{ and }\,
												\omega=\frac{1}{h\sqrt{2}}\left(
												\begin{array}{c}
															z^{3}-2\\
															z^{2}
												\end{array}\right),
									\end{equation*}					
									we may take
									\begin{equation*}
												\Upsilon=\frac{1}{h}\left(
												\begin{array}{c}
															-z\\
															1-z^{3}
												\end{array}\right)\;\mbox{ and }\,
												\Omega=\frac{1}{h\sqrt{2}}\left(
												\begin{array}{c}
															-z^{2}\\
															z^{3}-2
												\end{array}\right),
									\end{equation*}	
									so that the matrix functions
									\[	\cV=(\,\nu\;\bar{\Upsilon}\,)\,\mbox{ and }\;\cW^{t}=(\,\omega\;\bar{\Omega}\,)	\]
									are thematic.  Since
									\[	\Omega^{*}G\bar{\Upsilon}=\frac{\bar{z}^{4}}{\bar{h}^{2}}(3-\bar{z}^{3}-z^{3})
											=\frac{\bar{z}^{4}}{\bar{h}^{2}}|h|^{2}=\bar{z}^{4}\frac{h}{\bar{h}},	\]
									$G$ admits the factorization
									\begin{equation*}
												G=\cW^{*}\left(
																\begin{array}{cc}
																			\bar{z}\frac{\bar{h}}{h}	&	\Zero\\
																			\Zero											&	\bar{z}^{4}\frac{h}{\bar{h}}
																\end{array}\right)\cV^{*}
									\end{equation*}
									with sequence of thematic indices $(1,4)$ as desired.
			\end{enumerate}
\end{ex}

\textbf{Acknowledgments.}  This paper is based in part on the author's Ph.D. dissertation at Michigan State University.
The author would like to thank his advisor, Professor Vladimir V. Peller, for his support, patience, helpful remarks and 
valuable discussions concerning this work.

\end{document}